\documentclass[11pt]{amsproc}

\usepackage[margin=3cm]{geometry}

\newtheorem{theorem}{Theorem}

\newtheorem{lemma}{Lemma}

\begin{document}

\title{The distribution functions of $\sigma(n)/n$ and $n/\varphi(n)$, II}
\author{Andreas Weingartner}
\address{Department of Mathematics, Southern Utah University, Cedar City, Utah 84720}
\email{weingartner@suu.edu}

\maketitle

\section{Introduction}

Let
\begin{equation*}
A(t):= \lim_{N\rightarrow \infty}
\frac{1}{N} \left|\left\{n\le N: \sigma(n)/n \ge t \right\}\right|,
\end{equation*}
where $\sigma(n)$ is the sum of the positive divisors of $n$,
and
\begin{equation*}
B(t):= \lim_{N\rightarrow \infty}
\frac{1}{N} \left|\left\{n\le N: n/\varphi(n) \ge t \right\}\right|,
\end{equation*}
where $\varphi$ denotes Euler's totient function. 
Both of these limits exist and are continuous functions of $t$ \cite{Davenport, Schoenberg}.

We are interested in the size of $A(t)$ and $B(t)$
as $t$ tends to infinity.
From the work of Erd\H{o}s \cite{Erdos} it follows that
\begin{equation*}
B(t) = \exp\left\{ -e^{t \, e^{-\gamma}}
\left(1+o(1)\right) \right\} \qquad (t\rightarrow \infty),
\end{equation*}
which was sharpened and extended to $A(t)$ by the author \cite{Wei07} with the result
\begin{equation}\label{old}
A(t), B(t) = \exp\left\{ -e^{t \, e^{-\gamma}}
\left(1+O\left({t^{-2}}\right)\right) \right\} \qquad (t\rightarrow \infty)
\end{equation}
where $\gamma=0.5772...$ is Euler's constant.

The purpose of this note is to make further improvements to the error term. 

\begin{theorem}\label{thm1}
We have
\begin{equation*}
A(t), B(t) = \exp \left\{- e^{t \, e^{-\gamma}}\left(1 + \sum_{j=2}^m \frac{a_j}{t^j}+ O_m\left(\frac{1}{t^{m+1}}\right)\right)\right\},
\end{equation*}
where
\begin{equation*}
a_2 = -\frac{\pi^2}{6}\, e^{2\gamma}, \quad a_3= \frac{\pi^2}{6}\,e^{3\gamma}, \quad a_4=-\left(\frac{\pi^2}{6}+\frac{37\pi^4}{360}\right)\, e^{4\gamma}.
\end{equation*}
\end{theorem}

Additional coefficients $a_i$ can be determined without major difficulties by following the proofs of Lemma \ref{Wz}, Lemma \ref{Wy} and Section \ref{secthm}, starting with the coefficients $b_i$ from Lemma \ref{Wz}. 

Throughout we will use the notation
\begin{equation}\label{ydef}
y= y(t):=  e^{t \, e^{-\gamma}}.
\end{equation} 

We can further decrease the size of the error term in Theorem \ref{thm1} in exchange for a more complex main term. Let
\begin{equation}\label{Idef}
I(y,s) :=\int_e^{y} \log \left(1 + x\, e^{-s/x} \right) \frac{dx}{\log x} 
+ \int_{y}^{y\log y} \log \left(1 + x^{-1}\, e^{s/x}\right)\frac{dx}{\log x},
\end{equation} 
and
\begin{equation*}
 L(y) := \exp \left\{ \frac{(\log y)^{3/5}}{(\log\log y)^{1/5}}\right\}.
\end{equation*} 

\begin{theorem}\label{thm2}
There exists a positive constant $c$ such that
\begin{equation*}
A(t), B(t) = \exp \left\{-y +  \min_{s \in J} I(y,s) + R(y)\right\},
\end{equation*}
where $J=[y\log y - y, y\log y +y]$ and 
\begin{equation*}
R(y) = O\left(\frac{y}{L(y)^c}\right) .
\end{equation*}
Assuming the Riemann hypothesis we have
\begin{equation*}
R(y) = O\left(\sqrt{y} \, (\log y)^2 \right).
\end{equation*}
\end{theorem}

The behavior of $B(t)$ near $t=1$ is described by Tenenbaum and Toulmonde \cite[Thm. 1.2]{TenToul}, who show that
\begin{equation}\label{TT}
1-B(1+1/(\sigma -1)) = \sum_{j=1}^m \frac{g_j}{(\log \sigma)^j}+ O\left(\frac{|g_{m+1}|}{(\log \sigma)^{m+1}} + \frac{1}{L(\sigma)^c}\right),
\end{equation}
for some $c>0$, where
\begin{equation*}
g_1 = e^{-\gamma}, \quad g_2= 0, \quad g_3=-\frac{1}{12}\, \pi^2 \, e^{-\gamma},
\end{equation*}
and
\begin{equation*}
g_j = \{1+O(j^{-1})\} \, e^{-\gamma} (-1)^{j+1} (j-3)! \qquad (j\ge 3) .
\end{equation*}

A classic result (see e.g. \cite{Schoenberg}) states that for all $s\in \mathbb C$ we have
\begin{equation}\label{W}
    W(s) := \lim_{N\rightarrow \infty} \frac{1}{N}
    \sum_{n\le N} \left(\frac{n}{\varphi(n)}\right)^s
    =\prod_{p}\left(1 + \frac{(1-p^{-1})^{-s}-1}{p }\right)
\end{equation}
and thus
\begin{equation}\label{Mellin}
  \int_0^\infty B(x) \, x^{s-1} dx = 0  -\frac{1}{s}\int_0^\infty x^s \, dB(x) = \frac{W(s)}{s}, \quad (\Re(s)>0).
\end{equation}
Hence $\frac{W(s)}{s}$ is the Mellin transform of $B(t)$. The method used in \cite{TenToul} to establish \eqref{TT} is essentially that of inversion of the Mellin transform with the abscissa of integration moved to $-\sigma$. 
For large $t$ on the other hand, we find that $W(s)\, t^{-s}$ is small when $\Re(s)$ is close to $y\log y$.
It turns out that the minimum of $W(s)\, t^{-s}$ with respect to $s$ along the positive real axis is already an excellent estimate for $B(t)$ (see Lemma \ref{BW}), and it appears that inversion is not a natural choice in this case because of the slower convergence of the product in \eqref{W} when $\Re(s)>0$. Therefore we will restrict our investigation to $s \in [0, \infty)$.

The following result shows that $A(t)$ and $B(t)$ are close enough so that it suffices to show that Theorems \ref{thm1} and \ref{thm2} hold for $B(t)$, which is the simpler
object since $\varphi(n)$ does not depend on the multiplicities of the prime factors of $n$.
\begin{theorem}\label{AB}
For $t\ge t_0$ we have
\begin{equation*}
 A(t) \le B(t) <  e^{ 3\sqrt{y}} \, A\left(  t - \frac{5 e^\gamma}{\sqrt{y}}\right)
\end{equation*}
\end{theorem}

Another arithmetic function closely related to $\varphi$ and $\sigma$ is Dedekind's $\psi$ function, defined by
\begin{equation*}
\psi(n) =n\prod_{p|n} (1+p^{-1}).
\end{equation*}
With
\begin{equation*}
D(t):= \lim_{N\rightarrow \infty}
\frac{1}{N} \left|\left\{n\le N: \psi(n)/n \ge t \right\}\right|,
\end{equation*}
one can show that $D(t/\zeta(2))$ also satisfies Theorems \ref{thm1} and \ref{thm2}.
It is easy to see that $D(t/\zeta(2)) \ge B(t)$ using the definition of $\psi$ and $\varphi$.
For the upper bound of $D(t/\zeta(2))$ one can consider the analog of Lemma \ref{BW} (i) below.

\section{Proof of Theorem \ref{AB}}

The inequality $A(t)\le B(t)$, valid for all $t$, follows from
\begin{equation*}
\frac{\sigma(n)}{n}
=\prod_{p^\nu || n}\frac{1+p+\ldots + p^{\nu}}{p^{\nu}}
=\prod_{p^\nu || n} \frac{1-p^{-\nu-1}}{1-p^{-1}}
< \prod_{p|n} \frac{1}{1-p^{-1}}
=\frac{n}{\varphi(n)}.
\end{equation*}
To establish the second inequality of Theorem \ref{AB} we let 
\begin{equation*}
m=m(t)=\prod_{p\le \sqrt{y}} p^{h_p}, \quad \mbox{where} \ h_p =  \left\lfloor \frac{\log y}{\log p} \right\rfloor.
\end{equation*}
For every $n$ that satisfies 
\begin{equation*}
\frac{n}{\varphi(n)} = \prod_{p|n}\frac{1}{1-p^{-1}} \ge t,
\end{equation*}
$n m$ will satisfy
\begin{equation*}
\frac{\sigma(n m)}{n m} = \prod_{p^k || n m} \frac{1-p^{-k-1}}{1-p^{-1}} 
= \prod_{p || n m} \frac{1}{1-p^{-1}} \prod_{p^k || n m} (1-p^{-k-1})
\ge t P,
\end{equation*}
where
\begin{equation*}
P=\prod_{p^k || n m} (1-p^{-k-1}) 
\ge \prod_{p\le \sqrt{y}} \left(1-\frac{1}{y}\right) \prod_{p > \sqrt{y}} \left(1-\frac{1}{p^2}\right)
\ge 1 -\frac{5}{\sqrt{y}\log y},
\end{equation*}
for $t\ge t_0$, by a standard application of the prime number theorem. Thus
\begin{equation*}
\frac{\sigma (n m)}{n m} \ge t \left(1-\frac{5}{\sqrt{y}\log y}\right)= t - \frac{5 e^\gamma}{\sqrt{y}},
\end{equation*}
which implies
\begin{equation*}
A\left( t - \frac{5 e^\gamma}{\sqrt{y}}\right) \ge \frac{1}{m} \, B(t).
\end{equation*}
The result now follows since, for $t\ge t_0$, 
\begin{equation*}
\log m = \sum_{p\le \sqrt{y}} h_p \log p \le \sum_{p\le \sqrt{y}} \log y < 3\sqrt{y}.
\end{equation*}

\section{The relation between $B(t)$ and $W(s)$.}

\begin{lemma}\label{sylogy}
Let $s\ge 1$. If 
$$
 B(t) t^{s-1} = \max_{x\ge 0}  B(x) x^{s-1}
$$
then
$$
s = y \log y +O(y).
$$
\end{lemma}

\begin{proof}
Assume $B(t) t^{s-1}\ge B(t+h) (t+h)^{s-1}$ for $|h|\le 1$. After taking logarithms we use \eqref{old} to obtain
\begin{equation*}
y\, (e^{h\,e^{-\gamma}} - 1) + O(y\,t^{-2}) \ge (s-1) \log (1+h\, t^{-1}) ,
\end{equation*}
and hence
\begin{equation*}
y \, h\, e^{-\gamma}  \ge (s-1)h\, t^{-1} + O(s\, h^2\,t^{-2} + y h^2 + y\,t^{-2}) .
\end{equation*}
The result now follows if we first let $h=t^{-1}$, and then $h=-t^{-1}$, and multiply the last inequality by $h^{-1} t$ in each case. 
\end{proof}

\begin{lemma}\label{BW}
\begin{enumerate}
\item[(i)]For all $s\ge 0$, $t > 0$ we have
\begin{equation*}
B(t) \le  \frac{W(s)}{t^s}.
\end{equation*}
\item[(ii)] Let $s\ge 1$ and $t\ge t_0$. If 
$
B(t)t^{s-1} = \max_{x\ge 0} B(x)x^{s-1},
$
then
\begin{equation*}
 \frac{W(s)}{3 s \, t^s} \le B(t)
\end{equation*}
and 
\begin{equation*}
\log B(t) =  O(t) + \min_{u\ge 0} \log \frac{W(u)}{t^u}.
\end{equation*}
\end{enumerate}
\end{lemma}

\begin{proof}
(i) For all $s\ge 0$,
\begin{equation*}
    B(t) = \lim_{N \rightarrow \infty} \frac{1}{N}
    \sum_{\stackrel{n\le N}{n\ge t\,\varphi(n)}} 1 \le \lim_{N \rightarrow \infty} \frac{1}{N}
    \sum_{n\le N} \left(\frac{n}{t\, \varphi(n)}\right)^s
   =\frac{W(s)}{t^s}.
\end{equation*}

(ii) From \eqref{Mellin} we have
\begin{equation*}
\frac{W(s)}{s} = \int_0^{t} B(x) \, x^{s-1} dx + \int_t^{2t} B(x) \, x^{s-1} dx + \int_{2t}^\infty B(x) \, x^{s-1} dx =: I_1+I_2+I_3
\end{equation*}
Since $\max_{x\ge 0} (B(x) \, x^{s-1}) = B(t) \, t^{s-1}$, we have $I_1, I_2 \le t \, B(t) \, t^{s-1} =B(t) \, t^{s}$.
If $c$ is the implied constant in the error term of \eqref{old}, then for $x \ge t$
\begin{equation*}
\begin{split}
B(2x) &\le \exp\left\{ -e^{2x e^{-\gamma}} (1-c x^{-2})\right\}
 \le \exp\left\{ -e^{x e^{-\gamma}} (y/2)(1+c x^{-2})\right\} \\
& \le B(x)^{y/2} \le B(x)^{1+\log y} \le B(x) B(t)^{\log y} \\
& \le B(x) \exp\left\{ -y\log y +O(y/\log y)\right\} = B(x) \exp\left\{ -s+O(y)\right\} \\
& \le \frac{B(x)}{2^{s+1}},
\end{split}
\end{equation*}
since $s = y \log y +O(y)$ by Lemma \ref{sylogy}. 
We conclude that for $k\ge 1$ 
\begin{equation*}
 \int_{t2^k }^{t2^{k+1}} B(x) \, x^{s-1} dx 
 =2^s \int_{t2^{k-1} }^{t2^{k}} B(2x) \, x^{s-1} dx 
 \le \frac{1}{2} \int_{t2^{k-1}}^{t2^k} B(x) \, x^{s-1} dx,
\end{equation*}
and thus $I_3 \le I_2 \le B(t) \, t^{s}$.

The second assertion in (ii) follows from the first and (i), since $s = y \log y +O(y)$.
\end{proof}

\section{The study of the product $W(s)$.}

Let
\begin{equation*}\label{tuPu}
t_u := \prod_{p\le u} \frac{1}{1-p^{-1}}, \quad P_u := \prod_{p\le u} p .
\end{equation*}

\begin{lemma}\label{WF}
Let $2 \le u \le v$. For $s\ll v$ we have
 \begin{equation*}
\frac{W(s)}{t_u^s}= 
\frac{t_u}{t_v \, P_u} \left(1+O\left(\frac{s}{v\log v} \right)\right) 
\prod_{p\le u}\left(1+p\left(1-p^{-1}\right)^{s+1}\right) 
\prod_{u < p\le v} \left(1 + p^{-1}(1-p^{-1})^{-s-1}\right)
\end{equation*}
\end{lemma}

\begin{proof}
The contribution from primes $p>v$ to the product \eqref{W} is
\begin{equation*}
\begin{split}
   \prod_{p>v}\left(1 + \frac{(1-p^{-1})^{-s}-1}{p }\right)
   & = \prod_{p>v}\left(1 + \frac{1}{p}
  \left(e^{O\left(\frac{s}{p}\right)}-1\right)\right)\\
   & =\prod_{p>v}\left(1 + O\left(\frac{s}{p^2}\right)\right)
   =1 + O\left( \frac{s}{v \log v}\right).
\end{split}
\end{equation*}
For primes $p$ in the range $u < p\le v$ we write
\begin{equation*}
\prod_{u < p\le v} \left(1 + \frac{(1-p^{-1})^{-s}-1}{p }\right)
 = \prod_{u < p\le v} (1-p^{-1}) \prod_{u < p\le v} \left(1 + p^{-1}(1-p^{-1})^{-s-1}\right) .
\end{equation*}
Finally, the product over small primes is
\begin{equation*}
\prod_{p\le u} \frac{(1-p^{-1})^{-s}}{p}
\prod_{p\le u}\left(1+p\left(1-p^{-1}\right)^{s+1}\right) 
=\frac{t_u^s}{P_u}
\prod_{p\le u}\left(1+p\left(1-p^{-1}\right)^{s+1}\right) .
\end{equation*}
\end{proof}

\begin{lemma}\label{We}
Let $2 \le u \le v$. For $v\gg s = u\log u + O(u)$ we have
 \begin{equation*}
\frac{W(s)}{t_u^s}= 
\frac{t_u}{t_v \, P_u} \left(1+O\left(\frac{1}{\log u} \right)\right) 
\prod_{p\le u} \left(1 + p\, e^{-s/p} \right) 
\prod_{u < p\le v} \left(1 + p^{-1}\, e^{s/p}\right) .
\end{equation*}
\end{lemma}

\begin{proof}
We write
\begin{equation*}
\begin{split}
\prod_{u < p\le v} \left(1 + p^{-1}(1-p^{-1})^{-s-1}\right) = 
\prod_{u < p\le v} \left(1 + p^{-1}\exp\left(\frac{s}{p} +
  O\left(\frac{s}{p^2}\right)\right)\right)\\
 = \prod_{u < p\le v} \left(1 + p^{-1}e^{s/p}\right) 
 \left(1 + O\left(\frac{s}{p^3} e^{s/p}\right)\right).
\end{split}
\end{equation*}

After taking the logarithm of the last expression, the contribution from the error term is
\begin{equation}\label{primesum1}
\begin{split}
\ll \sum_{p> u } \frac{s}{p^3} e^{s/p} 
& \asymp \int_{u}^\infty \frac{s}{x^3} e^{s/x}\frac{dx}{\log x} \\
& \asymp \frac{1}{ u \log u}  \int_{u}^\infty \frac{s}{x^2} e^{s/x} dx \\
& = \frac{1}{ u \log u} e^{s/u} \asymp \frac{1}{ \log u} .
\end{split}
\end{equation}
Thus
\begin{equation*}
\prod_{u < p\le v} \left(1 + p^{-1}(1-p^{-1})^{-s-1}\right) = 
\left(1+O\left(\frac{1}{\log u}\right)\right)
\prod_{u < p\le v} \left(1 + p^{-1}e^{s/p}\right) .
\end{equation*}

Similarily,
\begin{equation*}
\begin{split}
\prod_{ p\le u} \left(1 + p\, (1-p^{-1})^{s+1}\right) = 
\prod_{p\le u} \left(1 + p\, \exp\left(-\frac{s}{p} +
  O\left(\frac{s}{p^2}\right)\right)\right)\\
 = \prod_{p\le u} \left(1 + p\, e^{-s/p}\right) 
 \left(1 + O\left(\frac{s}{p} \, e^{-s/p}\right)\right).
\end{split}
\end{equation*}
The contribution from the error term to the logarithm of the last expression is
\begin{equation}\label{primesum2}
\begin{split}
\asymp \sum_{p\le u } \frac{s}{p} \, e^{-s/p} 
& \asymp \int_2^{u} \frac{s}{x} \, e^{-s/x}\frac{dx}{\log x} \\
& \asymp \frac{u}{\log u}  \int_{2}^{u} \frac{s}{x^2}\, e^{-s/x} dx \\
& \asymp \frac{u}{ \log u} \, e^{-s/u} \asymp \frac{1}{\log u}.
\end{split}
\end{equation}
Thus
\begin{equation*}
\prod_{p\le u} \left(1 + p\, (1-p^{-1})^{s+1}\right) = 
\left(1+O\left(\frac{1}{\log u}\right)\right)
\prod_{p\le u} \left(1 + p\, e^{-s /p}\right) .
\end{equation*}
The result now follows from Lemma \ref{WF}.
\end{proof}

\begin{lemma}\label{Wz}
Let $s\ge e$ and define $z$ by $s=z\log z$. For $m\ge 2$ we have
\begin{equation*}
W(s) = \exp \left( z \log z \log(e^\gamma \log z) -z +z\sum_{j=2}^m \frac{b_j}{(\log z)^j}+ O_m\left(\frac{z}{(\log z)^{m+1}}\right)\right),
\end{equation*}
where
\begin{equation*}
b_2 = \frac{\pi^2}{6}, \quad b_3= -\frac{\pi^2}{6}, \quad b_4=\frac{\pi^2}{6}+\frac{7\pi^4}{60}.
\end{equation*}
\end{lemma}

\begin{proof}
We apply Lemma \ref{We} with $u=z$ and $v=s$ to obtain
\begin{equation}\label{WI}
\begin{split}
\log W(s) = & -s \sum_{p\le z}\log(1-p^{-1})  - \sum_{p\le z} \log p  +O\left(\frac{\log_2 z}{\log z} \right)\\
& \ +\sum_{p\le z} \log \left(1 + p\, e^{-s/p} \right) 
+ \sum_{z < p\le s} \log \left(1 + p^{-1}\, e^{s/p}\right) \\
= & z \log z \log(e^\gamma \log z) -z +O\left(\frac{z}{\exp(\sqrt{\log z})} \right) \\
 & \  + \int_e^{z} \log \left(1 + x\, e^{-s/x} \right) \frac{dx}{\log x} 
+ \int_{z}^{s} \log \left(1 + x^{-1}\, e^{s/x}\right)\frac{dx}{\log x},
\end{split}
\end{equation}
by a strong form of Mertens' Theorem \cite{Vino} and a standard application of the prime number theorem.
We need to estimate the two integrals in \eqref{WI}. The first integral is 
\begin{equation}\label{logsum}
\sum_{k\ge 1} \frac{(-1)^{k+1}}{k} \int_e^{z}  x^k\, e^{-sk/x}\frac{dx}{\log x} = \sum_{k\ge 1} \frac{(-1)^{k+1}}{k} I_k(k,1),
\end{equation}
where
\begin{equation}\label{Ikdef}
\begin{split}
I_k(a,b) & := \int_e^{z}  x^a\, e^{-sk/x}\frac{dx}{(\log x)^b} 
=\frac{1}{sk}\int_e^{z} \left( \frac{sk}{x^2} \, e^{-sk/x} \right) \frac{x^{a+2}}{(\log x)^b} \, dx \\
& \le \frac{z^{a+2}}{sk (\log z)^b} \int_e^{z} \frac{sk}{x^2} \, e^{-sk/x} dx 
\le \frac{z^{1+a-k}}{k (\log z)^{b+1}},
\end{split}
\end{equation}
for $a\ge b$, since $x/\log x$ is increasing for $x\ge e$. Integration by parts applied to the second integral in 
\eqref{Ikdef} shows that 
\begin{equation}\label{recursion}
I_k(a,b) =  \frac{z^{1+a-k}}{k (\log z)^{b+1}} - \frac{a+2}{sk}I_k(a+1,b)+\frac{b}{sk} I_k(a+1,b+1) + O_m(1/(sk)),
\end{equation}
for  $a\le k+m$. After $m-1$ iterations of \eqref{recursion}, starting with $I_k(k,1)$, we find that
\begin{equation}\label{Ik}
I_k(k,1) = \sum_{j=2}^m \frac{z}{(\log z)^j} \, q_j(k) + O_m\left(\frac{z}{k (\log z)^{m+1}}\right),
\end{equation}
where $q_j(k)$ is a rational function of $k$ with $q_j(k) = O(1/k)$. In particular, 
\begin{equation*}
q_2(k) = \frac{1}{k}, \quad q_3(k)= -\frac{k+2}{k^2}, \quad q_4(k)=\frac{1}{k^2}+\frac{(k+2)(k+3)}{k^3}.
\end{equation*}
Inserting \eqref{Ik} into \eqref{logsum} gives
\begin{equation}\label{Iresult}
\int_e^{z} \log \left(1 + x\, e^{-s/x} \right) \frac{dx}{\log x} 
= z \, \sum_{j=2}^m \frac{\theta_j}{(\log z)^j}  + O_m\left(\frac{z}{(\log z)^{m+1}}\right),
\end{equation}
where 
\begin{equation*}
\theta_j = \sum_{k\ge 1 } (-1)^{k+1} \frac{q_j(k)}{k} .
\end{equation*}
Similarily, the second integral in \eqref{WI} is 
\begin{equation}\label{Jlogsum}
\sum_{k\ge 1} \frac{(-1)^{k+1}}{k} \int_z^{s}  x^{-k}\, e^{sk/x}\frac{dx}{\log x} = \sum_{k\ge 1} \frac{(-1)^{k+1}}{k} J_k(k,1),
\end{equation}
where
\begin{equation}\label{Jkdef}
\begin{split}
J_k(a,b) & := \int_z^{s}  x^{-a}\, e^{sk/x}\frac{dx}{(\log x)^b} 
=\frac{1}{sk}\int_z^{s} \left( \frac{sk}{x^2} \, e^{sk/x} \right) \frac{x^{-a+2}}{(\log x)^b} \, dx \\
& \asymp \frac{1}{sk (\log z)^b} \int_k^{k\log z} e^w \, \left(\frac{sk}{w}\right)^{2-a}  dw
=O_m\left(\frac{z^{1+k-a}}{k (\log z)^{b+1}}\right),
\end{split}
\end{equation}
for $a\ge -m$.
Integration by parts applied to the second integral in \eqref{Jkdef} shows that 
\begin{equation}\label{Jrecursion}
J_k(a,b) =  \frac{z^{1+k-a}}{k (\log z)^{b+1}} + \frac{2-a}{sk}J_k(a-1,b)-\frac{b}{sk} J_k(a-1,b+1) + O_m\left(k^{-1}(e/s)^{a-1}\right),
\end{equation}
for $k-a \le m$. After $m-1$ iterations of \eqref{Jrecursion}, starting with $J_k(k,1)$, we find that
\begin{equation}\label{Jk}
J_k(k,1) = \sum_{j=2}^m \frac{z}{(\log z)^j} \, r_j(k) + O_m\left(\frac{z}{k (\log z)^{m+1}}\right),
\end{equation}
where $r_j(k)$ is a rational function of $k$ with $r_j(k) = O(1/k)$. In particular, 
\begin{equation*}
r_2(k) = \frac{1}{k}, \quad r_3(k)= \frac{2-k}{k^2}, \quad r_4(k)=\frac{(2-k)(3-k)}{k^3}-\frac{1}{k^2}.
\end{equation*}
Inserting \eqref{Jk} into \eqref{Jlogsum} gives
\begin{equation}\label{Jresult}
\int_z^{s} \log \left(1 + x^{-1}\, e^{s/x} \right) \frac{dx}{\log x} 
= z \, \sum_{j=2}^m \frac{\rho_j}{(\log z)^j}  + O_m\left(\frac{z}{(\log z)^{m+1}}\right),
\end{equation}
where
\begin{equation*}
\rho_j = \sum_{k\ge 1} (-1)^{k+1} \frac{r_j(k)}{k} .
\end{equation*}
Let $b_j=\theta_j + \rho_j$, then
\begin{equation*}
b_2 = \sum_{k\ge 1} \frac{(-1)^{k+1}}{k} (q_2(k)+r_2(k)) = 2 \sum_{k\ge 1} \frac{(-1)^{k+1}}{k^2} 
=2 \sum_{k\ge 1} \frac{1}{k^2} - 4  \sum_{k\ge 1} \frac{1}{(2k)^2} 
=\frac{\pi^2}{6},
\end{equation*}
\begin{equation*}
b_3 = \sum_{k\ge 1} \frac{(-1)^{k+1}}{k} (q_3(k)+r_3(k)) = -2 \sum_{k\ge 1} \frac{(-1)^{k+1}}{k^2} = -\frac{\pi^2}{6} ,
\end{equation*}
and
\begin{equation*}
b_4 = \sum_{k\ge 1} \frac{(-1)^{k+1}}{k} (q_4(k)+r_4(k)) =  \sum_{k\ge 1} \frac{(-1)^{k+1}}{k}\left(\frac{2}{k}+\frac{12}{k^3}\right) =
\frac{\pi^2}{6} +\frac{7\pi^4}{60} .
\end{equation*}
The result now follows from combining \eqref{WI}, \eqref{Iresult} and \eqref{Jresult}.
\end{proof}

\begin{lemma}\label{Wy}
For $t\ge 1$ and $y=  e^{t \, e^{-\gamma}}$ we have
\begin{equation*}
\min_{s\ge e} \frac{W(s)}{t^s} = \exp \left(-y +y\sum_{k=2}^m \frac{c_k}{(\log y)^k}+ O_m\left(\frac{y}{(\log y)^{m+1}}\right)\right),
\end{equation*}
where
\begin{equation*}
c_2 = \frac{\pi^2}{6}, \quad c_3= -\frac{\pi^2}{6}, \quad c_4=\frac{\pi^2}{6}+\frac{37\pi^4}{360}.
\end{equation*}
\end{lemma}

\begin{proof}
Let $t\ge 1$ be given. From Lemma \ref{Wz} we have 
\begin{equation}\label{hdef}
\log \frac{W(s)}{t^s} = z\left( \log z \log(\log z/\log y) -1 +\sum_{k=2}^m \frac{b_k}{(\log z)^k}+ O_m\left(\frac{1}{(\log z)^{m+1}}\right)\right)
=:h(z),
\end{equation}
where $s=z\log z$. 
We see that $h(y) \sim -y$ and $h(z)>0$ for $z\ge e y$, so that the minimum of $h(z)$ occurs at some $z\in [e, ey]$, where the error term
of \eqref{hdef} is uniformly $O_m\left(y/(\log y)^{m+1}\right)$. Therefore we only need to minimize
\begin{equation}\label{fz}
f(z) := z\left( \log z \log(\log z/\log y) -1 +\sum_{k=2}^m \frac{b_k}{(\log z)^k}\right).
\end{equation}
To that end we set $f'(z)=0$, which is equivalent to 
\begin{equation}\label{alpha}
\log y = \log z \exp \left( \sum_{k=2}^{m+1} \frac{\alpha_k}{(1+\log z) \log^k z} \right),
\end{equation}
where $\alpha_2 = b_2$, $\alpha_k = b_k-(k-1)b_{k-1}$ for $k=3, \ldots, m-1$, and $\alpha_{m+1}=-m b_m$. Thus
\begin{equation*}
\alpha_2 = b_2 = \frac{\pi^2}{6}, \quad \alpha_3=b_3-2b_2 = -\frac{\pi^2}{2}, \quad \alpha_4=b_4-3b_3=\frac{2\pi^2}{3}+\frac{7\pi^4}{60}.
\end{equation*}
Since $f(e) \sim -e\log\log y$, $f(y)\sim -y$, and $f(ey)>0$, the unique
solution to \eqref{alpha} is the minimizer of $f(z)$. We rewrite \eqref{alpha} as
\begin{equation}\label{beta}
\log z \log \left(\frac{\log z}{\log y}\right) = - \sum_{k=2}^{m} \frac{\beta_k}{(\log z)^{k}} + O\left(\frac{1}{(\log z)^{m+1}}\right),
\end{equation}
where $\beta_2=\alpha_2$ and $\beta_k=\alpha_k-\beta_{k-1}$ for $k=3, \ldots, m$. Thus
\begin{equation*}
\beta_2 = \alpha_2 = \frac{\pi^2}{6}, \quad \beta_3=\alpha_3-\beta_2 = -\frac{2\pi^2}{3}, \quad \beta_4=\alpha_4-\beta_3=\frac{4\pi^2}{3}+\frac{7\pi^4}{60}.
\end{equation*}
To express $z$ in terms of $y$ we first write \eqref{beta} as 
\begin{equation}\label{delta}
\begin{split}
\log y & = \log z \exp \left( \sum_{k=2}^{m} \frac{\beta_k}{(\log z)^{k+1}} + O\left(\frac{1}{(\log z)^{m+2}}\right)\right) \\
& =  \log z \left( 1 + \sum_{k=2}^{m} \frac{\delta_k}{(\log z)^{k+1}} + O\left(\frac{1}{(\log z)^{m+2}}\right)\right),
\end{split}
\end{equation}
where 
\begin{equation*}
\delta_2=\beta_2 = \frac{\pi^2}{6}, \quad \delta_3=\beta_3 = -\frac{2\pi^2}{3}, \quad \delta_4 = \beta_4=\frac{4\pi^2}{3}+\frac{7\pi^4}{60}.
\end{equation*}
Using series inversion on \eqref{delta} we obtain
\begin{equation}\label{eta}
\log z = \log y \left( 1 + \sum_{k=2}^{m} \frac{\eta_k}{(\log y)^{k+1}} + O\left(\frac{1}{(\log y)^{m+2}}\right)\right),
\end{equation}
where 
\begin{equation*}
\eta_2 = -\delta_2= -\frac{\pi^2}{6}, \quad \eta_3=-\delta_3=\frac{2\pi^2}{3}, \quad \eta_4 = -\delta_4 = -\frac{4\pi^2}{3}-\frac{7\pi^4}{60}.
\end{equation*}
We exponentiate \eqref{eta} to get
\begin{equation}\label{lambda}
\begin{split}
z & = y \exp \left( \sum_{k=2}^{m} \frac{\eta_k}{(\log y)^{k}} + O\left(\frac{1}{(\log y)^{m+1}}\right)\right) \\
& =  y\left( 1 + \sum_{k=2}^{m} \frac{\lambda_k}{(\log y)^{k}} + O\left(\frac{1}{(\log y)^{m+1}}\right)\right),
\end{split}
\end{equation}
where
\begin{equation*}
\lambda_2=\eta_2= -\frac{\pi^2}{6}, \quad \lambda_3 = \eta_3=\frac{2\pi^2}{3}, \quad \lambda_4 = \eta_4 +\frac{\eta_2^2}{2}=-\frac{4\pi^2}{3}-\frac{37\pi^4}{360}.
\end{equation*}
Combining \eqref{fz}, \eqref{beta} and \eqref{lambda} we see that $\min_{z}f(z)$ is
\begin{multline*}\label{mu}
y\left( 1 + \sum_{k=2}^{m} \frac{\lambda_k}{\log^k y} + O\left(\frac{1}{\log^{m+1} y}\right)\right) 
\left( -1 + \sum_{k=2}^{m} \frac{b_k-\beta_k}{\log^k z} + O\left(\frac{1}{\log^{m+1} z}\right)\right) \\
 = y\left( 1 + \sum_{k=2}^{m} \frac{\lambda_k}{(\log y)^{k}} + O\left(\frac{1}{(\log y)^{m+1}}\right)\right) 
\left( -1 + \sum_{k=2}^{m} \frac{\mu_k}{(\log y)^{k}} + O\left(\frac{1}{(\log y)^{m+1}}\right)\right),
\end{multline*}
where \eqref{delta} implies
\begin{equation*}
\mu_2 = b_2-\beta_2 = 0, \quad \mu_3 = b_3-\beta_3 = \frac{\pi^2}{2}, \quad \mu_4 = b_4-\beta_4 = -\frac{7\pi^2}{6}.
\end{equation*}
Thus
\begin{equation*}
\min_{z} f(z)=-y +y\sum_{k=2}^m \frac{c_k}{(\log y)^k}+ O_m\left(\frac{y}{(\log y)^{m+1}}\right),
\end{equation*}
where
\begin{equation*}
c_2 = \mu_2-\lambda_2 = \frac{\pi^2}{6}, \quad c_3=\mu_3-\lambda_3= -\frac{\pi^2}{6}, \quad c_4=\mu_4+\mu_2\lambda_2-\lambda_4=\frac{\pi^2}{6}+\frac{37\pi^4}{360}.
\end{equation*}
\end{proof}

\begin{lemma}\label{g2}
Let $g(s)=\log W(s)$. We have
\begin{equation*}
g''(s) \asymp \frac{1}{s\log s} \qquad (s\ge 2).
\end{equation*}
\end{lemma}

\begin{proof}
From \eqref{W} we find that
\begin{equation*}
\begin{split}
g''(s) & = \sum_{p} p^{-1}(1-p^{-1}) \log^2(1-p^{-1}) \frac{(1-p^{-1})^{-s}}{(1+p^{-1}((1-p^{-1})^{-s}-1))^2} \\
& \asymp \sum_{p} p^{-3} \frac{(1-p^{-1})^{-s}}{(1+p^{-1}(1-p^{-1})^{-s})^2} 
\end{split}
\end{equation*}
Let $u$ be given by $u^{-1}(1-u^{-1})^{-s} = 1$, so that $s= u \log u + O(\log u)$. Then 
\begin{equation*}
\begin{split}
g''(s) & \asymp \sum_{p\le u} p^{-1} (1-p^{-1})^{s} +  \sum_{p> u} p^{-3} (1-p^{-1})^{-s} \\
& \asymp \sum_{p\le u} p^{-1} e^{-s/p} +  \sum_{p> u} p^{-3} e^{s/p} \asymp \frac{1}{s \log s},
\end{split}
\end{equation*}
where the last two sums are estimated just like in \eqref{primesum1} and \eqref{primesum2}.
\end{proof}

\section{Proof of Theorems \ref{thm1} and \ref{thm2}}\label{secthm}

Define the set of maximizers
\begin{equation*}
M := \{t \ge 1: \exists s >1 \mbox{ with }  \max_{x\ge 0} B(x) x^{s-1} = B(t) t^{s-1} \}.
\end{equation*}
\begin{lemma}\label{gap}
There is a constant $c>0$ such that for every $t \ge 1$ there is a $t_1 \in M$ with 
\begin{equation*}
|t-t_1| \le c \sqrt{t/y}. 
\end{equation*}
\end{lemma}

\begin{proof}
Let $t\ge 1$ and let $s$ be given by $\min_u \frac{W(u)}{t^u}= \frac{W(s)}{t^s}$. Let $t_1 \in M$ satisfy
$\max_{x\ge 1} B(x) x^{s-1} = B(t_1)t_1^{s-1}$. Finally, define $s_1$ by $\min_u \frac{W(u)}{t_1^u}= \frac{W(s_1)}{t_1^{s_1}}$.
From Lemma \ref{BW} we find
\begin{equation*}
\frac{W(s)}{3s t_1^s} \le B(t_1) \le \frac{W(s_1)}{t_1^{s_1}} \le \frac{W(s)}{ t_1^s} ,
\end{equation*}
so
\begin{equation*}
\log \frac{W(s)}{t_1^s} = \log \frac{W(s_1)}{t_1^{s_1}} + O(\log s).
\end{equation*}
By Taylor's theorem there is an $s_0$ between $s$ and $s_1$ with
\begin{equation*}
\log \frac{W(s)}{t_1^s} = \log \frac{W(s_1)}{t_1^{s_1}} + \frac{g''(s_0)}{2} (s-s_1)^2,
\end{equation*}
where $g(u)=\log W(u)$. Combining the last two equations with Lemma \ref{g2} we obtain
\begin{equation*}
|s-s_1| = O(\sqrt{s} \log s).
\end{equation*}

Let $f(u)=\exp(g'(u))$. From the definition of $s$ and $s_1$ we have 
$t=f(s)$ and $t_1=f(s_1)$. Thus $|t-t_1| \le |s-s_1| \max_I f'(u)$, where $I$ is 
the interval with endpoints $s$, $s_1$. Now $f'(u)=f(u)\,g''(u)$, so Lemma \ref{g2} yields
\begin{equation*}
|t-t_1| =O\left(\frac{t \sqrt{s}\log s}{s \log s}\right) = O\left(\sqrt{t/y}\right),
\end{equation*}
by Lemma \ref{sylogy}.
\end{proof}

\begin{proof}[Proof of Theorem \ref{thm1}]
If $t \in M$ then the result follows from Lemma \ref{BW} (ii) and Lemma \ref{Wy} with $a_j = -c_j e^{j\gamma}$.
If $t \notin M$, the result follows from Lemma \ref{gap} and the monotonicity of $B(t)$.
\end{proof}

\begin{proof}[Proof of Theorem \ref{thm2}]
We apply Lemma \ref{We} with $u=y$ and $v=y\log y$. For $s=y\log y +O(y)$,
\begin{equation*}
\begin{split}
\log \frac{W(s)}{t^s} = & -s \log t -s \sum_{p\le y}\log(1-p^{-1})  - \sum_{p\le y} \log p  +O\left(\frac{\log_2 y}{\log y} \right)\\
& \ +\sum_{p\le y} \log \left(1 + p\, e^{-s/p} \right) 
+ \sum_{y < p\le y \log y} \log \left(1 + p^{-1}\, e^{s/p}\right) \\
= & -y + I(y,s) + O\left(\frac{y}{L(y)^c} \right)  ,
\end{split}
\end{equation*}
by a strong form of Mertens' Theorem \cite{Vino} and a standard application of the prime number theorem.
Under the assumption of the Riemann hypothesis, the error term can be replaced by $O(\sqrt{y} \, (\log y)^2)$. 
If $t\in M$ and $\max_{x\ge 0} B(x) x^{s-1} = B(t) t^{s-1}$, then $s=y\log y + o(y)$, by an argument 
like in the proof of Lemma \ref{sylogy}, but this time using Theorem \ref{thm1} with $m=2$ instead of \eqref{old}.
Therefore  Lemma \ref{BW} implies
\begin{equation*}
\log B(t) = O(t) + \min_{u\in J} \frac{W(u)}{t^u} = O(t) -y +  \min_{u\in J} I(y,u) +  O\left(\frac{y}{L(y)^c} \right) ,
\end{equation*}
where $J=[y\log y - y, y\log y +y]$.
If $t \notin M$, there is a $t_1 \in M$ with $|t-t_1| = O(\sqrt{t/y})$ by Lemma \ref{gap}.
For $s=y\log y +O(y)$ and $y_1 := e^{t_1 \, e^{-\gamma}}=y+O(\sqrt{t\,y})$ we have $ I(y_1,s)=I(y,s) + O(\sqrt{t\,y})$.
Thus the result follows again from the monotonicity of $B(t)$.
\end{proof}


\begin{thebibliography}{00}

\bibitem{Davenport}
H. Davenport, \"Uber numeri abundantes, \textit{Preuss. Akad. Wiss. Sitzungsber.} (1933), 830-837.

\bibitem{Erdos}
P. Erd\H{o}s, Some remarks about additive and multiplicative functions,
\textit{Bull. Amer. Math. Soc.} \textbf{52} (1946), 527-537.

\bibitem{Schoenberg}
I. J. Schoenberg, \"{U}ber die asymptotische Verteilung reeller Zahlen mod 1,
\textit{Math. Z.} \textbf{28} (1928), 171-199.

\bibitem{TenToul}
G. Tenenbaum, V. Toulmonde, Sur le comportement local de la r\'epartition de l'indicatrice
d'Euler, \textit{Funct. Approx. Comment. Math.} \textbf{35} (2006), 321-338. 


\bibitem{Vino}
A. I. Vinogradov, On the remainder in Mertens' formula. \textit{Dokl. Akad. Nauk SSSR}
\textbf{148} (1963), 262-263.

\bibitem{Wei07}
A. Weingartner, The distribution functions of $\sigma(n)/n$ and $n/\varphi(n)$,
\textit{Proc. Amer. Math. Soc.} \textbf{135} (2007), 2677-2681.


\end{thebibliography}
\end{document}